\documentclass[11pt,reqno,sumlimits]{amsart}
\usepackage{blindtext, hyperref, amssymb,amscd, epsfig, mathrsfs, xypic, amsmath,amsthm, color, enumerate, eepic} 
\xyoption{all}
\usepackage{fancyhdr}
\usepackage{colortbl}
\usepackage{arydshln}
\newcommand{\bbT}{{\mathbb T}}

\newcommand{\ZZ}{{\mathbb Z}}

\newcommand{\frakG}{{\mathfrak G}}
\newcommand{\frakS}{{\mathfrak S}}

\newcommand{\calF}{{\mathcal F}}

\newcommand{\scA}{{\mathscr A}}

\newcommand{\FSVT}{{SVT}}

\newtheorem{thm}{Theorem}[section]  
 
\newtheorem{lem}[thm]{Lemma}  
\newtheorem{prop}[thm]{Proposition} 
\newtheorem{df-pr}[thm]{Definition-Proposition}

\theoremstyle{definition}

\newtheorem{exm}[thm]{Example}

\numberwithin{equation}{section} 
\begin{document} 
\title{A tableau formula of double Grothendieck polynomials for $321$-avoiding permutations}
\author{Tomoo Matsumura}

\pagestyle{plain}
\lhead{T. Matsumura}
\rhead{}

\date{\today}
\maketitle 
\begin{abstract}
In this article, we prove a tableau formula for the double Grothendieck polynomials associated to $321$-avoiding permutations. The proof is based on the compatibility of the formula with the $K$-theoretic divided difference operators.
\end{abstract}
\section{Introduction}

Let $S_n$ be the permutation group of $\{1,\dots, n\}$ and a permutation $w\in S_n$ is called {\it $321$-avoiding} if there are no numbers $i<j<k$ such that $w(i)>w(j)>w(k)$. The Grassmannian permutations are also examples of such permutations. The goal of this paper is to prove a tableau formula of the {\it double Grothendieck polynomials} $G_w(x,b)$ associated to those $321$-avoiding permutations $w$. 

Lascoux and Sch{\"u}tzenberger (\cite{LascouxSchutzenberger3}, \cite{Lascoux1}) introduced the (double) Grothendieck polynomials which are polynomial representatives of the (equivariant) $K$-theory classes of structure sheaves of Schubert varieties in a full flag variety. Fomin and Kirillov (\cite{GrothendieckFomin}, \cite{DoubleGrothendieckFomin}) further studied them in the terms of $\beta$-Grothendieck polynomials and Yang-Baxter equations, obtaining their combinatorial formula. For the case of {\it Grassmannian permutations} , or more generally {\it vexillary permutations} (i.e. $2143$-avoiding), the corresponding double Grothendieck polynomials are expressed in terms of set-valued tableaux by the work of Buch \cite{BuchLRrule}, McNamara \cite{McNamara}, and Knutson--Miller--Yong \cite{KnutsonMillerYong}. On the other hand, the author jointly with Hudson, Ikeda and Naruse obtained a determinant formula of double Grothendieck polynomials for Grassmannian permutations (\cite{HIMN}, see \cite{Anderson2016} and \cite{HudsonMatsumura2} for vexillary case), in the context of degeneracy loci formulas in $K$-theory. Recently, Anderson--Chen--Tarasca \cite{AndersonChenTarasca} extended the method in \cite{HIMN} to the case of $321$-avoiding permutations in the study of Brill-Noether loci in $K$-theory, and obtained a determinant formula of the corresponding double Grothendieck polynomials. Combined with the work \cite{MatsumuraFlagged} by the author on skew Grothendieck polynomials, it implies a tableaux formula of (single) Grothendieck polynomials associated to 321 avoiding permutations. See also the most recent related work \cite{ChanPflueger} by Chan--Pflueger. This new development motivated our work in this paper.

Let $w\in S_n$ be a $321$-avoiding permutation. Following the work of Billey--Jockusch--Stanley \cite{BilleyJockuschStanley} (cf. Chen--Yan--Yang \cite{ChenYanYang} and Anderson--Chen--Tarasca \cite{AndersonChenTarasca}), one can find a skew partition $\sigma(w)=\lambda/\mu$ and a flagging $f(w)$ which is a subsequence of $(1,\dots, n)$. Namely, let $f(w)=(f_1,\dots, f_r)$ be the increasing sequence of numbers $i$ such that $w(i)>i$. We define the partitions $\lambda$ and $\mu$ by
\[
\lambda_i=w(f_r) - r -f_i+i, \ \ \ \mu_i=w(f_r) - r - w(f_i)+i, \ \ \ \ \ i=1,\dots,r.
\]
As mentioned above, it is known that the single Grothendieck polynomials $\frakG_w(x)$ has the following tableau formula
\[
\frakG_w(x) = \sum_{T \in \FSVT(\sigma(w),f(w))} \beta^{|T|-|\sigma(w)|} \prod_{e\in T}x_{val(e)}, 
\]
where $\FSVT(\sigma(w),f(w))$ is the set of set-valued tableaux of shape $\sigma(w)$ with flagging $f(w)$. On the other hand, Chen--Yan--Yang \cite{ChenYanYang} showed that the double Schubert polynomials $\frakS_w(x,b)$ of Lascoux and Sch{\"u}tzenberger associated to $w$ has the formula
\[
\frakS_w(x,b) = \sum_{T}\prod_{e\in T} (x_{val(e)}+ b_{\lambda_{r(e)}+f_{r(e)} - c(e) - val(e) + 1}),
\]
where $T$ runs over the set of all semistandard tableaux of shape $\sigma(w)$ with flagging $f(w)$. Note that both fomulas specialize to the formula of the single Schubert polynomial $\frakS_w(x)$ found earlier by Billey--Jockusch--Stanley \cite{BilleyJockuschStanley}.

Our main theorem (Theorem \ref{thmmain}) unifies the two formulas above. We show that the double Grothendieck polynomial $\frakG_w(x,b)$ associated to $w$ is given by
\[
\frakG_w(x,b) = \sum_{T \in \FSVT(\sigma(w),f(w))} \beta^{|T|-|\sigma(w)|}\prod_{e\in T}(x_{val(e)}\oplus b_{\lambda_{r(e)}+f_{r(e)} - c(e) - val(e) + 1}),
\]
where $u\oplus v:=u+v + \beta uv$ for any variables $u$ and $v$.

The main ingredient for the proof is Proposition \ref{propmain}, the compatibility of the tableau formula with the $K$-theoretic divided difference operators which are used to define Grothendieck polynomials. In proving this compatibility, we closely follow the argument used by Wachs in \cite[Lemma 1.1]{Wachs}. 

The paper is organized as follows. In $\S$ \ref{S1}, we recall the definitions of the divided difference operators and the double Grothendieck polynomials, as well as a few basic formulas. In $\S$ \ref{S2}, we recall the definitions of set-valued tableaux of skew shape with flagging and describe our main theorem (Theorem \ref{thmmain}). Grassmannian permutations are basic examples of $321$-avoiding permutations and we also prove the theorem in this case (Lemma \ref{thmGr}). In $\S$ \ref{S3}, we show a key proposition and then prove our main theorem by induction.

\section{Divided difference and double Grothendieck polynomials}\label{S1}
Let $x=(x_1,x_2,\cdots,x_n)$ and $b=(b_1,b_2,\dots,b_n)$ be two sets of variables. Let $\ZZ[\beta][x,b]$ be the polynomial ring in variable $x$ and $b$ over the coefficient ring $\ZZ[\beta]$ where $\beta$ is a formal variable of degree $-1$. Let $S_n$ be the permutation group of $\{1,\dots, n\}$ and $s_i=(i,i+1), \ i=1,\dots, n-1$, denote the transpositions that generates $S_n$. The length of a permutation $w$, denoted by $\ell(w)$, is the number of transpositions in the reduced word decomposition of $w$. Let the action of $S_n$ on $\ZZ[\beta][x,b]$ be defined by
\[
\sigma(f(x_1,\dots, x_n)) := f(x_{\sigma(1)},\dots, x_{\sigma(n)}).
\]
For each $i \in \ZZ_{>0}$, we define the divided difference operator $\pi_i$ on $\ZZ[\beta][x,b]$ by
\[
\pi_i(f) = \frac{(1+\beta x_{i+1}) f - (1+\beta x_i ) s_if }{x_i-x_{i+1}}
\]
for each polynomial $f \in \ZZ[\beta][x,b]$.

The {\it double Grothendieck polynomial} $\frakG_w=\frakG_w(x,b) \in \ZZ[\beta][x,b]$ associated to a permutation $w\in S_n$ is defined inductively as follows. Our convention coincides with the one in \cite{BuchLRrule} after setting $\beta=-1$. For any variable $u$ and $v$, we use the notation $u\oplus v:=u+v+\beta uv$. For the longest element $w_0$ in $S_n$, we set
\[
\frakG_{w_0} = \prod_{i+j\leq n}  (x_i\oplus b_j).
\]
If $w$ is not the longest, we can find a positive integer $i$ such that $\ell(ws_i)=\ell(w) +1$ and then define
\[
\frakG_w:= \pi_i(\frakG_{ws_i}).
\]
This definition is independent of the choice of $s_i$ because the operators $\pi_i$ satisfy the Coxeter relations. 

To conclude this section, we recall a few basic formulas for $\pi_i$  (cf.  \cite[\S2.1]{MatsumuraFlagged}). Let $f, g\in \ZZ[\beta][[x,b]]$. First, we have the following Leibniz rule for $\pi_i$: 
\begin{equation}\label{ddo1}
\pi_i(fg) = \pi_i(f)g + s_i(f)\pi_i(g) + \beta s_i(f)g.
\end{equation}
If $f$ is symmetric in $x_i$ and $x_{i+1}$, then we have
\begin{eqnarray}
\pi_i(f) &=& -\beta f,  \label{ddo2}\\
\pi_i(fg) &=&f\pi_i(g), \label{ddo3}
\end{eqnarray}
and moreover we have
\begin{equation}\label{x^n}
\pi_i(x_i^k f) = 
\begin{cases}
-\beta f& (k=0), \\
\left(\displaystyle\sum_{s=0}^{k-1} x_i^s x_{i+1}^{k-1-s}  + \beta \displaystyle\sum_{s=1}^{k-1} x_i^{s} x_{i+1}^{k-s}\right)f & (k>0).
\end{cases}
\end{equation}

\section{Flagged skew partitions and $321$-avoiding permutations}
\subsection{Definitions and the main theorem}\label{S2}
Let $\lambda=(\lambda_1\geq \cdots \geq \lambda_r>0)$ and $\mu=(\mu_1\geq \cdots \geq \mu_r\geq 0)$ be partitions such that $\mu_i\leq \lambda_i$ for all $i=1,\dots,r$. We identify a partition with its Young diagram. A {\it skew partition} $\lambda/\mu$ is given by the pair of $\lambda$ and $\mu$. We identify a skew partition $\lambda/\mu$ with its {\it skew diagram} which is the collection of boxes in $\lambda$ that are not in $\mu$. More precisely, we denote $\lambda/\mu=\{(i,j) \ |\ \mu_i < j \leq \lambda_i, k=1,\dots, r \}$. When $\mu$ is empty, the skew partition $\lambda/\mu$ is regarded as the partition $\lambda$. Let $|\lambda/\mu|$ be the numbers of boxes in the corresponding skew diagram. 

For finite subsets $a$ and $b$ of positive integers, we define $a<b$ if $max(a)<min(b)$, and $a\leq b$ if $max(a)\leq min(b)$. A {\it set-valued tableau} $T$ of shape $\lambda/\mu$ is a labeling by which each box of the skew diagram $\lambda/\mu$ is assigned a finite subset of positive integers, called a {\it filling}, in such a way that the rows are weakly increasing from left to right and the columns strictly increasing from top to bottom. An element $e$ of a filling of $T$ is called an {\it entry} and denoted by $e \in T$. The numerical value of an entry $e \in T$ is denoted by $val(e)$, and the row and column indices of the box to which $e$ belongs are denoted by $r(e)$ and $c(e)$ respectively. Let $|T|$ be the total number of entries of $T$.

A {\it flagging} of $\lambda/\mu$ is a sequence of positive integers $f=(f_1,\dots, f_r)$ such that $f_1\leq \cdots \leq f_r$. A \emph{set-valued tableau of skew shape $\lambda/\mu$ with a flagging $f$} is a set-valued tableau of skew shape $\lambda/\mu$ such that each filling in the $i$-th row is a subset of $\{1,\dots,f_i\}$ for all $i$. Let $\FSVT(\lambda/\mu,f)$ denote the set of all skew tableaux of shape $\lambda/\mu$ with a flagging $f$. If $f_1=\cdots=f_r$, then the associated set-valued tableaux are nothing but the set-valued tableaux of skew shape $\lambda/\mu$ considered by Buch in \cite{BuchLRrule}. Note that in \cite{Wachs} and \cite{MatsumuraFlagged}, one considers more general flagged skew partitions and their set-valued tableaux.

A permutation $w \in S_n$ is called {\it $321$-avoiding} if there are no numbers $i<j<k$ such that $w(i)>w(j)>w(k)$. Such permutation $w$ is completely characterized by a pair of increasing subsequence of $(1,\dots, n)$ (cf. \cite[$\S$2]{ErikssonLinusson}). Namely, let $f(w)=(f_1,\dots, f_r)$ be the increasing sequence of indices $i$ such that $i<w(i)$ and then $h(w):=(w(f_1),\dots, w(f_r))$ is also an increasing sequence. If $f^c(w)=(f_1^c,\dots, f_{n-r}^c)$ is the increasing sequence of indices $i$ such that $i\geq w(i)$, then $h^c(w)=(w(f_1^c),\dots w(f_{n-r}^c))$ is also an increasing sequence. Therefore one can see that $f(w)$ and $h(w)$ determines $w$ uniquely.

To $w$ one assigns a skew partition $\sigma'(w)=\lambda'/\mu'$ used by Anderson--Chen--Tarasca \cite{AndersonChenTarasca}: 
\begin{eqnarray*}
\lambda_i':=
w(f_{r+1-i}) - (r+1-i),\ \ \ 
\mu_i':=f_{r+1-i}- (r+1 - i).
\end{eqnarray*}
The skew partition $\sigma(w)=\lambda/\mu$ defined by Billey--Jockusch--Stanley \cite{BilleyJockuschStanley} can be obtained from rotating $\sigma'(w)$ by $180$ degree  and it relates to $\sigma'(w)$ by
\begin{eqnarray}
\lambda_i&:=& 
\lambda_1' - \mu_{r+1-i}'= 
w(f_r)-r - (f_i-i),\label{lamdef}\\ 
\mu_i &:=& 
\lambda_1' - \lambda_{r+1-i}' = 
w(f_r)-r - (w(f_i)-i).\label{mudef}
\end{eqnarray}
We have $\ell(w)=|\sigma(w)|$. Note also that we have $\lambda_1' = w(f_r)-r=\lambda_i + f_i -i$ for all $i=1,\dots, r$. 

The following is the main theorem of this article.
\begin{thm}\label{thmmain}
Let $w\in S_n$ be a $321$-avoiding permutation and $\sigma(w)$ its associated skew partition with the flagging $f(w)$. Then we have
\begin{equation}\label{eqmain}
\frakG_{w}(x,b) = \sum_{T\in \FSVT(\sigma(w),f(w))} \beta^{|T|-|\sigma(w)|} \prod_{e\in T} x_{val(e)} \oplus b_{\lambda_{r(e)} + f_{r(e)} - c(e) - val(e) +1}.
\end{equation}
\end{thm}
\begin{exm}
Let $w=(31254)$. Then $f(w)=(1,4)$, $h(w)=(3,5)$, $f^c(w)=(2,3,5)$ and $h^c(w)=(1,2,4)$. We find that 
$\lambda'=
(3,2)$, $\mu'=
(2,0)$, $\lambda=
(3,1) 
$ and $\mu=
(1,0)
$. Below is a few examples of set-valued tableaux in $\FSVT(\sigma(w),f(w))$.
\setlength{\unitlength}{0.4mm}
\begin{center}
\begin{picture}(30,25)
\put(10,20){\line(1,0){20}}
\put(0,10){\line(1,0){30}}
\put(0,0){\line(1,0){10}}
\put(0,0){\line(0,1){10}}
\put(10,0){\line(0,1){20}}
\put(20,10){\line(0,1){10}}
\put(30,10){\line(0,1){10}}

\put(3,3){{\tiny  $1$}}
\put(13,13){{\tiny  $1$}}
\put(23,13){{\tiny  $1$}}
\end{picture}
\begin{picture}(30,25)
\put(10,20){\line(1,0){20}}
\put(0,10){\line(1,0){30}}
\put(0,0){\line(1,0){10}}
\put(0,0){\line(0,1){10}}
\put(10,0){\line(0,1){20}}
\put(20,10){\line(0,1){10}}
\put(30,10){\line(0,1){10}}

\put(3,3){{\tiny  $2$}}
\put(13,13){{\tiny  $1$}}
\put(23,13){{\tiny  $1$}}
\end{picture}
\begin{picture}(30,25)
\put(10,20){\line(1,0){20}}
\put(0,10){\line(1,0){30}}
\put(0,0){\line(1,0){10}}
\put(0,0){\line(0,1){10}}
\put(10,0){\line(0,1){20}}
\put(20,10){\line(0,1){10}}
\put(30,10){\line(0,1){10}}

\put(3,3){{\tiny  $3$}}
\put(13,13){{\tiny  $1$}}
\put(23,13){{\tiny  $1$}}
\end{picture}
\begin{picture}(30,25)
\put(10,20){\line(1,0){20}}
\put(0,10){\line(1,0){30}}
\put(0,0){\line(1,0){10}}
\put(0,0){\line(0,1){10}}
\put(10,0){\line(0,1){20}}
\put(20,10){\line(0,1){10}}
\put(30,10){\line(0,1){10}}

\put(3,3){{\tiny  $4$}}
\put(13,13){{\tiny  $1$}}
\put(23,13){{\tiny  $1$}}
\end{picture}
\begin{picture}(30,25)
\put(10,20){\line(1,0){20}}
\put(0,10){\line(1,0){30}}
\put(0,0){\line(1,0){10}}
\put(0,0){\line(0,1){10}}
\put(10,0){\line(0,1){20}}
\put(20,10){\line(0,1){10}}
\put(30,10){\line(0,1){10}}

\put(2,3){{\tiny  $12$}}
\put(13,13){{\tiny  $1$}}
\put(23,13){{\tiny  $1$}}
\end{picture}
\begin{picture}(30,25)
\put(10,20){\line(1,0){20}}
\put(0,10){\line(1,0){30}}
\put(0,0){\line(1,0){10}}
\put(0,0){\line(0,1){10}}
\put(10,0){\line(0,1){20}}
\put(20,10){\line(0,1){10}}
\put(30,10){\line(0,1){10}}

\put(2,3){{\tiny  $13$}}
\put(13,13){{\tiny  $1$}}
\put(23,13){{\tiny  $1$}}
\end{picture}
\begin{picture}(30,25)
\put(10,20){\line(1,0){20}}
\put(0,10){\line(1,0){30}}
\put(0,0){\line(1,0){10}}
\put(0,0){\line(0,1){10}}
\put(10,0){\line(0,1){20}}
\put(20,10){\line(0,1){10}}
\put(30,10){\line(0,1){10}}

\put(2,3){{\tiny  $14$}}
\put(13,13){{\tiny  $1$}}
\put(23,13){{\tiny  $1$}}
\end{picture}
\begin{picture}(30,25)
\put(10,20){\line(1,0){20}}
\put(0,10){\line(1,0){30}}
\put(0,0){\line(1,0){10}}
\put(0,0){\line(0,1){10}}
\put(10,0){\line(0,1){20}}
\put(20,10){\line(0,1){10}}
\put(30,10){\line(0,1){10}}

\put(0.3,3){{\tiny  $123$}}
\put(13,13){{\tiny  $1$}}
\put(23,13){{\tiny  $1$}}
\end{picture}
\begin{picture}(30,25)
\put(10,7){$\cdots$}
\end{picture}
\end{center}
If $T$ is the 5th one above, the corresponding term in the summation of (\ref{eqmain}) is
\[
\beta (x_1 \oplus b_{2} ) (x_1\oplus b_{1} )(x_1\oplus b_{4} )(x_2\oplus b_{3}).
\]
\end{exm}
\subsection{Grassmannian case}
In this section, we prove Theorem \ref{thmmain} in the Grassmannian case. A permutation $w\in S_n$ is called {\it Grassmannian with descent at $d$} if there is at most one descent at $d$, {\it i.e.} $w(1)<\cdots < w(d)$ and $w(d+1)<\cdots<w(n)$. By definition, a Grassmannian permutation is $321$-avoiding. In this case, $f(w)=(s+1,\dots, s+r)$ where $s+1$ is the smallest index $i$ such that $i<w(i)$ and $f_r=s+r=d$. Hence, for each $i=1,\dots, r$, we find that $\lambda'_i=w(d+1-i)-(r+1-i)$ and $\mu'_i=s$ so that $\bar\lambda(w):=\lambda'-\mu'$ is a partition. We can also find that $\lambda_i = w(d)-d$ and $\mu_i=w(d)-r - (w(s+i)-i)$. 
\begin{exm}
Consider $w=(13524)$ which is a Grassmannian permutation with descent at $d=3$. Then we have $f(w)=(2,3)$, $h(w)=(3,5)$, $\lambda'=
(3,2)
$,
$\mu'=(1,1)$, $\lambda=(2,2)$, $\mu
= (1,0)$, $\bar\lambda(w)=(2,1)$ and $\bar f(w)=(3,3)$.
\end{exm}
The double Grothendieck polynomial associated to a Grassmannian permutation $w$ is known to be a symmetric polynomial in $x_1,\dots, x_d$ with coefficients in $\ZZ[\beta][b]$ and it can be expressed by the following formula (see \cite{BuchLRrule} and \cite{McNamara}): let $\bar f =(d,\dots, d)$, and we have
\begin{equation}\label{eqGr}
\frakG_w(x,b) = \sum_{T \in \FSVT(\bar\lambda(w), \bar f)} \beta^{|T| - |\bar\lambda(w)|} \prod_{e\in T} (x_{val(e)}\oplus b_{val(e) + c(e)-r(e)}).
\end{equation}
\begin{lem}\label{thmGr}
If $w\in S_n$ is Grassmannian with descent at $d$, the equation (\ref{eqmain}) holds.
\end{lem}
\begin{proof}
First observe that there is a bijection 
\[
\FSVT(\sigma(w),f(w)) \to \FSVT(\bar \lambda(w), \bar f)
\]
sending $T$ to $T'$ which is obtained by changing the value $i$ of each entry in $T$ to $d+1-i$ and rotating it by $180$ degree. The inverse map can be defined by the same way. Although $T$ is a tableau with flagging $(1,2,\dots, r)$ and the numbers used in the $i$-th row of the tableau $T'$ seems bounded below by $i$,  this would not affect the bijection because of the column strictness of tableaux of the partition $\bar \lambda(w)$. 

If $e\in T$ corresponds to $e'\in T'$ under this bijection, we have
\[
val(e)= d+1-val(e'), \ \ \  r(e) = r+1-r(e'), \ \ \ c(e) = w(d)-d+1-c(e').
\]
Thus, we have
\begin{eqnarray*}
&&\sum_{T\in \FSVT(\sigma(w),f(w))} \beta^{|T|-|\sigma(w)|} \prod_{e\in T} x_{val(e)} \oplus b_{\lambda_{r(e)} + f_{r(e)} - c(e) - val(e) +1}\\
&=&\sum_{T\in \FSVT(\sigma(w),f(w))} \beta^{|T|-|\sigma(w)|} \prod_{e\in T} x_{val(e)} \oplus b_{w(d)-r +r(e) - c(e) - val(e) +1}\\
&=&\sum_{T'\in \FSVT(\bar \lambda(w), \bar f)} \beta^{|T'|-|\bar \lambda(w)|} \prod_{e'\in T'} x_{d+1-val(e')} \oplus b_{val(e')-r(e')+c(e')}\\
&=&\sum_{T'\in \FSVT(\bar \lambda(w), \bar f)} \beta^{|T'|-|\bar \lambda(w)|} \prod_{e'\in T'} x_{val(e')} \oplus b_{val(e')-r(e')+c(e')}\\
&=&\frakG_w(x,b),
\end{eqnarray*}
where the third equality follows from the fact that $\frakG_w(x,b)$ is a symmetric polynomials in $x_1,\dots, x_d$.
\end{proof}
\section{Proof of the main theorem}\label{S3}
\subsection{Preparation}
The following two lemmas will be used in the proof of Proposition \ref{propmain} which allows us to prove the main theorem by induction.
\begin{lem}\label{lemkey}
For an arbitrary finite sequence of positive integers $\ell=(\ell_1,\dots, \ell_k)$,  we have
\begin{eqnarray}
\pi_i\left((x_i \oplus b_{\ell_1}) \cdots (x_i \oplus b_{\ell_k})\right) 
&=& \sum_{j=1}^k \left(\prod_{v=1}^{j-1} (x_i \oplus b_{\ell_v}) \prod_{v=j+1}^k (x_{i+1} \oplus b_{\ell_v})\right)\nonumber\\
&& + \beta \sum_{j=1}^{k-1} \left(\prod_{v=1}^{j} (x_i \oplus b_{\ell_v}) \prod_{v=j+1}^k (x_{i+1} \oplus b_{\ell_v})\right) \label{eqkey}.
\end{eqnarray}
In particular, $\pi_i(x_i \oplus b_{\ell_1})=1$. Furthermore, the expression on the right hand side of (\ref{eqkey}) is symmetric in $x_i$ and $x_{i+1}$.
\end{lem}
\begin{proof}
It suffices to show the claim for $i=1$ and $\ell=(1,\dots,k)$:
\begin{eqnarray}
\pi_1\left((x_1 \oplus b_{1}) \cdots (x_1 \oplus b_{k})\right) 
&=& \sum_{j=1}^k \left(\prod_{v=1}^{j-1} (x_1 \oplus b_{v}) \prod_{v=j+1}^k (x_2 \oplus b_{v})\right)\nonumber\\
&& + \beta \sum_{j=1}^{k-1} \left(\prod_{v=1}^{j} (x_1 \oplus b_{v}) \prod_{v=j+1}^k (x_2 \oplus b_{v})\right)\label{eqkey2}.
\end{eqnarray}
Choose an integer $n>k$ and consider the Grassmannian permutation $w=(k+1,1,2,\cdots, k,k+2,\dots, n)$ with descent at $1$.
We have $\bar\lambda(w) = (k)$ and $\bar f(w)=1$. There is only one set-valued tableau of shape $\bar\lambda(w)$ with flagging $\bar f(w)$, which assigns $\{1\}$ to each box. Therefore it follows from (\ref{eqGr}) that
\[
\frakG_w(x,b) = (x_1 \oplus b_1) \cdots (x_1 \oplus b_k).
\]
Now consider the element $ws_1=(1,k+1,2,3,\dots, k,k+2,\dots,n)$. We have $\bar\lambda(ws_1) = (k-1)$ and $\bar f(w)=(2)$. Since $\ell(w)=\ell(ws_1)+1$, we find that
\[
\pi_1\frakG_w(x,b) = \frakG_{ws_1}(x,b).
\]
We can also see that $\frakG_{ws_1}(x,b)$ coincides with the right hand side of (\ref{eqkey2}) by (\ref{eqGr}). Therefore (\ref{eqkey2}) holds. The last claim follows from the fact that $\frakG_{ws_1}(x,b)$ is symmetric in $x_1$ and $x_2$.
\end{proof}
\begin{lem}\label{lemtabkey}
Let $w$ be a $321$-avoiding permutation. Let $f(w)=(f_1,\dots, f_r)$ and $\sigma(w)=\lambda/\mu$. Suppose that $f_i+1<f_{i+1}$ and $w(f_i)>f_i+1$ for some $i$. 
Then $ws_{f_i}$ is a $321$-avoiding permutation and we have
\begin{eqnarray*}
f(ws_{f_i})=(f_1,\dots, f_{i-1}, f_i+1, f_{i+1},\dots, f_r), \ && \ \ \ \ h(ws_{f_i})= h(w),\\
\lambda(ws_{f_i})= (\lambda_1,\dots, \lambda_{i-1},\lambda_{i}-1,\lambda_{i+1},\dots, \lambda_r), && \ \ \ \ \mu(ws_{f_i})=\mu(w).
\end{eqnarray*}
 In particular, $\ell(w)=\ell(ws_i)+1$.
\end{lem}
\begin{proof}
First observe that $f_i+1<f_{i+1}$ implies that $w(f_i+1) < f_i+1$. Indeed, if $w(k)=k$ for some $k$, we find that $w'=(w(1),\dots, w(k-1))$ is a permutation in $S_{k-1}$ since $w(f_j)>f_j$ for all $j$ and $h^c(w)$ is an increasing sequence.  Therefore $w(k-1)\leq k-1$. It follows that $w s_{f_i}(f_i) = w(f_i + 1) \leq f_i$ and $ws_{f_i}(f_i+1) = w(f_i)>f_i+1$. From this, we find that $f(ws_{f_i})$ and $h(ws_{f_i})$ are as given in the claim. The rest can be checked by computing $\lambda(ws_{f_i})$ and $\mu(ws_{f_i})$ from the definition (\ref{lamdef}) and (\ref{mudef}).
\end{proof}
The following proposition is the main ingredient of the proof of Theorem \ref{thmmain}.
\begin{prop}\label{propmain}
Let $w$ be a $321$-avoiding permutation. Let $f(w)=(f_1,\dots, f_r)$ and $\sigma(w)=\lambda/\mu$. Suppose that $f_i+1<f_{i+1}$ and $w(f_i)>f_i+1$ for some $i$. 
Let $t:=f_i$. Then we have
\[
\partial_{t} \bbT_w(x,b) = \bbT_{ws_t}(x,b),
\]
where $\bbT_w(x,b)$ is the right hand side of (\ref{eqmain}).
\end{prop}
\begin{proof}
We define an equivalence relation in $\FSVT(\sigma(w),f(w))$ as follows: $T_1\sim T_2$ if the collection of boxes containing $t$ and $t':=t+1$ is the same for $T_1$ and $T_2$. We can write
\[
\bbT_w=\bbT_w(x,b) = \sum_{\scA\in \FSVT(\sigma(w),f(w))/\sim} \left(\sum_{T\in \scA} \beta^{|T|-|\sigma(w)|}(x|b)^T\right),
\]
where we denote
\[
(x|b)^T
=\prod_{e\in T} x_{val(e)} \oplus b_{\lambda_{r(e)} + f_{r(e)} - c(e) - val(e) +1}
=\prod_{e\in T} x_{val(e)} \oplus b_{\lambda_1' + r(e) - c(e) - val(e) +1}.
\]
By (\ref{lamdef}), we see that the condition $f_i+1<f_{i+1}$ implies $\lambda_i>\lambda_{i+1}$. Let $\scA$ be the equivalence class for $\FSVT(\sigma(w),f(w))$ whose tableaux have the same configuration of $t$ and $t'$ as shown in Figure 1. 
\setlength{\unitlength}{0.5mm}\begin{center}\begin{picture}(210,120)\thicklines
\put(50,120){\line(1,0){165}}
\put(50,120){\line(0,-1){20}}
\put(30,100){\line(0,-1){20}}
\put(30,100){\line(1,0){20}}
\put(10,80){\line(0,-1){30}}
\put(10,80){\line(1,0){20}}
\put(0,50){\line(1,0){10}}

\put(0,50){\line(0,-1){50}}

\put(200,100){\line(1,0){15}}
\put(200,100){\line(0,-1){10}}
\put(190,90){\line(1,0){5}}
\put(215,120){\line(0,-1){20}}

\put(140,90){\line(1,0){60}}
\put(80,80){\line(1,0){120}}
\put(80,70){\line(1,0){90}}
\put(80,60){\line(1,0){30}}
\put(160,50){\line(1,0){30}}
\put(140,30){\line(1,0){20}}
\put(0,0){\line(1,0){140}}

\put(180,93){\small{$A_1$}}
\put(150,93){\small{$B_1$}}
\put(120,83){\small{$A_2$}}
\put(90,83){\small{$B_2$}}
\put(45,33){\small{$A_k$}}
\put(20,33){\small{$B_k$}}

\put(175,83){\small{$t t\dots t$}}
\put(203,83){\small{$\cdots i$-th row}}
\put(145,83){\small{$t \dots t$}}
\put(145,73){\small{$t'\dots t'$}}
\put(85,73){\small{$t \dots t$}}
\put(85,63){\small{$t'\dots t'$}}
\put(115,73){\small{$*\dots *$}}
\put(178,75){\tiny{$m_1$}}
\put(150,65){\tiny{$r_1$}}
\put(120,65){\tiny{$m_2$}}
\put(90,55){\tiny{$r_2$}}
\put(73,48){{$\cdot$}}
\put(68,43){{$\cdot$}}
\put(63,38){{$\cdot$}}
\put(39,23){\small{$*\dots *$}}
\put(15,23){\small{$t\dots t$}}
\put(13.5,13){\small{$t'\dots t'$}}
\put(160,13){\small{$\cdots (i+k)$-th row}}
\put(45,16){\tiny{$m_k$}}
\put(19.5,6){\tiny{$r_k$}}
\put(10,30){\line(1,0){50}}
\put(10,20){\line(1,0){50}}
\put(10,10){\line(1,0){25}}
\put(10,30){\line(0,-1){20}}
\put(35,30){\line(0,-1){20}}
\put(60,30){\line(0,-1){10}}
\put(200,90){\line(0,-1){10}}
\put(170,90){\line(0,-1){20}}
\put(140,90){\line(0,-1){20}}
\put(110,80){\line(0,-1){20}}
\put(80,80){\line(0,-1){20}}
\put(190,80){\line(0,-1){30}}
\put(160,50){\line(0,-1){20}}
\put(140,30){\line(0,-1){30}}
\end{picture}\\
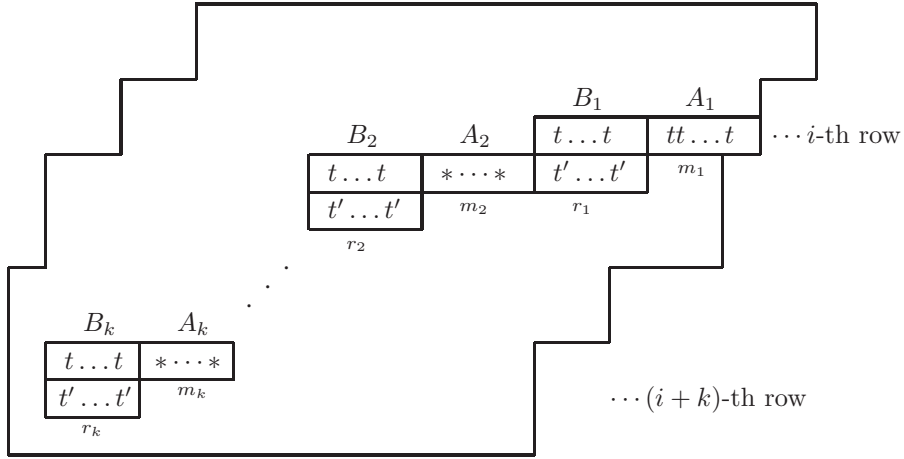
Figure 1.
\end{center}
In Figure 1, the rightmost one-row rectangle with entries $t$ has $m_1$ boxes and is denoted by $A_1$. For $s\geq 2$,  the $s$-th one-row rectangle with $*$, denoted by $A_s$, has $m_s$ boxes each of which contains $t$, $t'$, or both so that the total number of entries $t$ and $t'$ in $A_s$ is $m_s$ or $m_s+1$.  The $s$-th rectangle with two rows where the first row contains $t$ and the second row contains $t'$ has $r_s$ columns and is denoted by $B_s$. Note that $m_s$ and $r_s$ may be $0$ and hence the rectangles in Figure 1 may not be connected. Also the leftmost box of $A_s\  (s\geq 1)$ may contain a number less than $t$, the rightmost box of $A_s\ (s\geq 2)$ may contain a number greater than $t'$, and so on.  Let $a_s$ be the column index of the leftmost box in $A_s$. Similarly, let $b_s$ be the column index of the leftmost column in $B_s$. 

We can write
\[
\sum_{T\in \scA}\beta^{|T|-|\sigma(w)|}(x|b)^T = R(A_1)\cdot \prod_{s=2}^k R(A_s) \cdot \prod_{s=1}^k R(B_s) \cdot R(\scA)
\]
where $R(A_s)$ and $R(B_s)$ are the polynomials contributed from $A_s$ and $B_s$ respectively, and $R(\scA)$ is the polynomial contributed from the entries other than $t$ and $t'$. More precisely we have
\begin{eqnarray*}
R(A_1) &=& \prod_{\ell=a_1}^{a_1+m_1-1}(x_t \oplus b_{\lambda_1' + i - \ell - t +1}) =\prod_{v=1}^{m_1}(x_t \oplus b_{\lambda_1' + i - (a_1+v-1) - t +1})\\
R(A_s)  
&=&  \sum_{j=0}^{m_s} \left(
\prod_{\ell=a_s}^{a_s+j-1} (x_t \oplus b_{\lambda_1' + (i+s-1) - \ell - t +1}) 
\prod_{\ell=a_s+j+1}^{a_s+m_s} (x_{t'} \oplus b_{\lambda_1' + (i+s-1) - \ell - t +1})\right)\\
&& + \beta \sum_{j=1}^{m_s} \left(
\prod_{\ell=a_s}^{a_s+j-1} (x_t \oplus b_{\lambda_1' + (i+s-1) - \ell - t +1}) 
\prod_{\ell=a_s+j}^{a_s+m_s} (x_{t'} \oplus b_{\lambda_1' + (i+s-1) - \ell - t +1})\right) \ \ \ (2\leq s\leq k)\\
R(B_s)  &=& \prod_{\ell=b_s}^{b_s+r_s-1} (x_t \oplus b_{\lambda_1' + (i+s-1) - \ell - t +1})(x_{t'} \oplus b_{\lambda_1' + (i+s) - \ell - t' +1})  \ \ \ (1\leq s\leq k).
\end{eqnarray*}
Observe that $R(B_s)$ for $s\geq 1$ is symmetric in $x_{t}$ and $x_{t'}$ , and so is $R(A_s)$ for $s\geq 2$ by Lemma \ref{lemkey}. Therefore, if $m_1=0$, then $r_1=0$, $R(A_1)=0$ and by (\ref{ddo2}) we have 
\begin{equation}\label{eqA1}
\pi_t\left(\sum_{T\in \scA}\beta^{|T|-|\sigma(w)|}(x|b)^T\right) = -\beta \prod_{s=2}^k R(A_s) \cdot \prod_{s=2}^k R(B_s) \cdot R(\scA).
\end{equation}
If $m_1=1$, we have, by (\ref{ddo3}) and Lemma \ref{lemkey}, 
\begin{equation}\label{eqA2}
\pi_t\left(\sum_{T\in \scA}\beta^{|T|-|\sigma(w)|}(x|b)^T\right) = \prod_{s=2}^k R(A_s) \cdot \prod_{s=1}^k R(B_s) \cdot R(\scA).
\end{equation}
If $m_1\geq 2$, we have, also by (\ref{ddo3}) and Lemma \ref{lemkey}, 
\begin{eqnarray}\label{eqA3}
\pi_t\left(\sum_{T\in \scA}\beta^{|T|-|\sigma(w)|}(x|b)^T\right)
&=&\left(\sum_{j=1}^{m_1} \left(\prod_{v=1}^{j-1} (x_t \oplus b_{\ell_v}) \prod_{s=j+1}^{m_1} (x_{t'} \oplus b_{\ell_v})\right)\right.\nonumber\\
&&\ \ \ \ \  + \beta \left.\sum_{j=1}^{m_1-1} \left(\prod_{v=1}^{j} (x_t \oplus b_{\ell_v}) \prod_{v=j+1}^{m_1} (x_{t'} \oplus b_{\ell_v})\right)\right)\nonumber\\
&&\ \ \ \ \ \ \ \ \  \times \prod_{s=2}^k R(A_s) \cdot \prod_{s=1}^k R(B_s) \cdot R(\scA),
\end{eqnarray}
where $\ell_v={\lambda_1' + i - (a_1+v-1) - t +1}$.

%

We consider the decomposition
\[
\FSVT(\sigma(w),f(w))/\!\!\sim \ = \calF_1 \sqcup \calF_2\sqcup \calF_3\sqcup \calF_4
\]
where $\calF_1,\dots, \calF_4$ are the sets of equivalence classes whose configurations of the boxes containing $t$ or $t'$ satisfy the following conditions respectively:
\begin{itemize}
\item[(1)] $m_1=0$ (so that $r_1=0$),
\item[(2)] $m_1=1$ and the box at $(i,\lambda_i)$ in $\lambda$ contains more than one entry (so that $r_1=0$),
\item[(3)] $m_1=1$ and the box at $(i,\lambda_i)$ in $\lambda$ contains contains only $t$,
\item[(4)] $m_1\geq 2$.
\end{itemize}
Observe that there is a bijection from $\calF_2$ to $\calF_1$ sending $\scA_2$ to $\scA_1$ by deleting $t$ in the rectangle $A_1$ and that $R(\scA_2) = \beta R(\scA_1)$. Therefore,  by the expressions (\ref{eqA1}) and (\ref{eqA2}), we have
\[
\sum_{\scA \in \calF_1\sqcup \calF_2}\pi_t\left(\sum_{T \in \scA}  \beta^{|T| - |\sigma(w)|}(x|b)^T\right) = 0.
\]
Thus we obtain
\[
\pi_t(\bbT_w(x,b)) = \sum_{\scA\in \calF_3 \sqcup\calF_4} \pi_t\left(\sum_{T\in \scA} \beta^{|T|-|\sigma(w)|}(x|b)^T\right).
\]

On the other hand, by Lemma \ref{lemtabkey}, the skew partition $\sigma(ws_t)$ is obtained from $\sigma(w)$ by deleting the rightmost box of the $i$-th row and the flagging $f(ws_t)$ is obtained from $f(w)$ by adding $1$ to $f_i$. Thus for each equivalence class $\scA'$ for $\FSVT(\sigma(ws_t),f(ws_t))$, the corresponding tableaux have the following configuration of $t$ and $t'$ in Figure 2 where $m_1\geq 1$. 
\setlength{\unitlength}{0.5mm}
\begin{center}
\begin{picture}(210,130)
\thicklines
\put(50,120){\line(1,0){165}}
\put(50,120){\line(0,-1){20}}
\put(30,100){\line(0,-1){20}}
\put(30,100){\line(1,0){20}}
\put(10,80){\line(0,-1){30}}
\put(10,80){\line(1,0){20}}
\put(0,50){\line(1,0){10}}

\put(0,50){\line(0,-1){50}}

\put(200,100){\line(1,0){15}}
\put(200,100){\line(0,-1){10}}
\put(190,90){\line(1,0){5}}
\put(215,120){\line(0,-1){20}}

\put(140,90){\line(1,0){60}}
\put(80,80){\line(1,0){110}}
\put(80,70){\line(1,0){90}}
\put(80,60){\line(1,0){30}}
\put(160,50){\line(1,0){30}}
\put(140,30){\line(1,0){20}}
\put(0,0){\line(1,0){140}}
\put(171,83){\small{$*\dots *$}}
\put(145,83){\small{$t \dots t$}}
\put(145,73){\small{$t'\dots t'$}}
\put(85,73){\small{$t \dots t$}}
\put(85,63){\small{$t'\dots t'$}}
\put(115,73){\small{$*\dots *$}}
\put(172,75){\tiny{$m_1\!-\!1$}}
\put(150,65){\tiny{$r_1$}}
\put(120,65){\tiny{$m_2$}}
\put(90,55){\tiny{$r_2$}}
\put(73,48){{$\cdot$}}
\put(68,43){{$\cdot$}}
\put(63,38){{$\cdot$}}
\put(39,23){\small{$*\dots *$}}
\put(15,23){\small{$t\dots t$}}
\put(13.5,13){\small{$t'\dots t'$}}
\put(45,16){\tiny{$m_k$}}
\put(19.5,6){\tiny{$r_k$}}
\put(10,30){\line(1,0){50}}
\put(10,20){\line(1,0){50}}
\put(10,10){\line(1,0){25}}
\put(10,30){\line(0,-1){20}}
\put(35,30){\line(0,-1){20}}
\put(60,30){\line(0,-1){10}}
\put(190,90){\line(0,-1){10}}
\put(170,90){\line(0,-1){20}}
\put(140,90){\line(0,-1){20}}
\put(110,80){\line(0,-1){20}}
\put(80,80){\line(0,-1){20}}
\put(190,80){\line(0,-1){30}}
\put(160,50){\line(0,-1){20}}
\put(140,30){\line(0,-1){30}}
\end{picture}\\
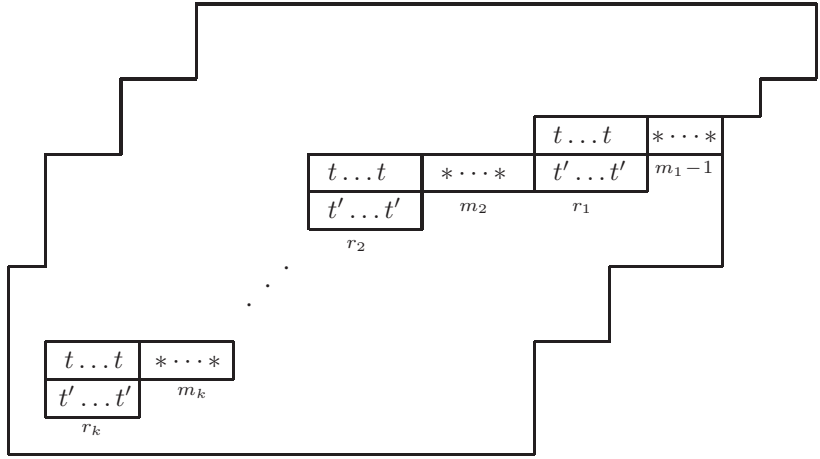
Figure 2.
\end{center}
Consider the decomposition 
\[
\FSVT(\sigma(ws_t),f(ws_t))/\!\!\sim\ = \calF_3' \sqcup \calF_4'
\]
where $\calF_3'$ is the set of equivalence classes such that $m_1=1$ and $\calF_4'$ is the set of equivalence classes such that $m_1\geq 2$. Obviously there are bijections $\calF_3 \to \calF_3'$ and $\calF_4\to \calF_4'$. Namely, $\scA\in \calF_3$ corresponds to $\scA'\in \calF_3$ by removing the single $t$ in the rectangle $A_1$, and $\scA\in \calF_4$ corresponds to $\scA'\in \calF_4$ by removing the last box of $A_1$ with $t$ and replace all other $t$'s by $*$, without changing all other entries of the tableaux. 

Under these bijections, $\sum_{T'\in \scA'}\beta^{|T'|-|\sigma(ws_t)|}(x|b)^{T'}$ is exactly the right hand side of (\ref{eqA2}) if $\scA\in \calF_3$ and (\ref{eqA3}) if $\scA\in \calF_4$. Thus we have 
\[
\pi_t\left(\sum_{T\in \scA}\beta^{|T|-|\sigma(w)|}(x|b)^T\right) =\sum_{T'\in \scA'}\beta^{|T'|-|\sigma(ws_t)|}(x|b)^{T'}.
\]
Therefore
\begin{eqnarray*}
\pi_t(\bbT_w) 
&=& \sum_{\scA\in \calF_3} \pi_t\left(\sum_{T\in \scA} \beta^{|T|-|\sigma(w)|}(x|b)^T\right)+\sum_{\scA\in \calF_4} \pi_t\left(\sum_{T\in \scA} \beta^{|T|-|\sigma(w)|}(x|b)^T\right)\\
&=& \sum_{\scA'\in \calF_3'} \sum_{T'\in \scA'}\beta^{|T'|-|\sigma(ws_t)|}(x|b)^{T'}+\sum_{\scA'\in \calF_4'} \sum_{T'\in \scA'}\beta^{|T'|-|\sigma(ws_t)|}(x|b)^{T'}\\
&=& \bbT_{ws_t}.
\end{eqnarray*}
This completes the proof.
\end{proof}
\subsection{Proof of Theorem \ref{thmmain}}
Let $h:=(h_1,\dots, h_r)$ be an increasing finite sequence of positive integers such that $h_1\geq 2$ and $h_r\leq n$. Let $C_h$ be the set of the $321$-avoiding permutations $v$ in $S_n$ such that $h(v)=h$. The union of $C_h$'s where $h$ runs over the set of all such increasing sequences coincides with the set of all $321$-avoiding permutations in $S_n$. We can define a total order on $C_h$ by $v < v'$ if $f(v)<f(v')$ in the lexicographic order. We show that $\frakG_v= \bbT_v$ for each $v\in C_h$ by induction on this order. The minimum element in $C_h$ is the Grassmannian permutation $v^{(0)}$ where $f(v^{(0)})=(1,\dots, r)$ and $h(v^{(0)}) = h$, and we have $\frakG_{v^{(0)}}(x,b) = \bbT_{v^{(0)}}(x,b)$ by Lemma \ref{thmGr}. Let $v\in C_h$ and suppose that $\frakG_w=\bbT_w$ for all $w < v$. Let $f(v)=(f_1,\dots, f_r)$. If $v^{(0)}<v$, then there is an index $i$ such that $f_{i-1}<f_{i}-1$. Then $v=ws_t$ with $t:=f_i -1$ where $w \in C_h$ is defined by $f(w)=(f_1,\dots, f_{i-1}, f_i-1, f_{i+1},\dots, f_r)$. Since $w<v$, we have $\frakG_w=\bbT_w$. Furthermore, $\ell(w)=\ell(v)+1$ and $w$ satisfies the conditions in Proposition \ref{propmain}. Now it follows from Proposition \ref{propmain} that
\[
\frakG_v = \pi_t(\frakG_w) = \pi_t(\bbT_w) = \bbT_v.
\]
This completes the proof of Theorem \ref{thmmain}.
\qed
\begin{exm}
Let us demonstrate how Proposition \ref{propmain} implies Theorem \ref{thmmain} in examples. Let $h=(3,5)$. Then the minimum element in $C_h$ is the Grassmannian permutation $v^{(0)}=(35124)$ where $f(v^{(0)})=(1,2)$. Consider $v=(31254) \in C_h$ where $f(v)=(1,4)$. We have $vs_3s_2=v^{(0)}$, $\ell(v^{(0)}) = \ell(v)+2$, and hence $\frakG_v=\pi_3\pi_2(\frakG_{v^{(0)}})$. Let $w=v^{(0)}s_2=(31524)$ so that $f(w)=(1,3)$. The corresponding skew partitions are 
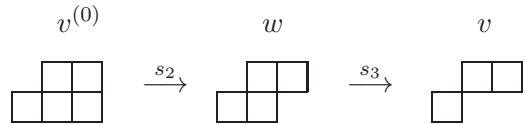
\setlength{\unitlength}{0.4mm}
\begin{center}
\begin{picture}(30,40)
\put(15,30){$v^{(0)}$}
\put(10,20){\line(1,0){20}}
\put(0,10){\line(1,0){30}}
\put(0,0){\line(1,0){30}}
\put(0,0){\line(0,1){10}}
\put(10,0){\line(0,1){20}}
\put(20,00){\line(0,1){20}}
\put(30,00){\line(0,1){20}}
\end{picture}
\begin{picture}(35,40)
\put(10,10){$\stackrel{s_2}{\longrightarrow}$}
\end{picture}
\put(15,30){$w$}
\begin{picture}(30,40)
\put(10,20){\line(1,0){20}}
\put(0,10){\line(1,0){30}}
\put(0,0){\line(1,0){20}}
\put(0,0){\line(0,1){10}}
\put(10,0){\line(0,1){20}}
\put(20,0){\line(0,1){20}}
\put(30,10){\line(0,1){10}}
\end{picture}
\begin{picture}(35,40)
\put(10,10){$\stackrel{s_3}{\longrightarrow}$}
\end{picture}
\begin{picture}(30,40)
\put(15,30){$v$}
\put(10,20){\line(1,0){20}}
\put(0,10){\line(1,0){30}}
\put(0,0){\line(1,0){10}}
\put(0,0){\line(0,1){10}}
\put(10,0){\line(0,1){20}}
\put(20,10){\line(0,1){10}}
\put(30,10){\line(0,1){10}}
\end{picture}
\end{center}
Since we know $\frakG_{v^{(0)}} = \bbT_{v^{(0)}}$ from Lemma \ref{thmGr}, Proposition \ref{propmain} implies
\[
\frakG_v =\pi_3\pi_2(\frakG_{v^{(0)}}) =\pi_3\pi_2(\bbT_{v^{(0)}}) = \pi_3(\bbT_{w}) = \bbT_v.
\]
%
%
%
\end{exm}
\bibliography{references}{}
\bibliographystyle{acm}

\end{document}